\documentclass[12pt,oneside]{amsart}
\usepackage{amscd,amsmath,amssymb,amsfonts,a4}
\newlength{\rulebreite}

\def\timesover#1#2#3{\ \xymatrix@1@=0pt@M=0pt{ _{#1}&\times&_{#2} \\& ^{#3}&}\ }
\def\otimesover#1#2#3{\ \xymatrix@1@=0pt@M=0pt{ _{#1}&\otimes&_{#2} \\& ^{#3}&}\ }
\usepackage[all]{xy}
\theoremstyle{plain}
\newtheorem{thm}{Theorem}
\newtheorem{lem}[thm]{Lemma}
\newtheorem{cor}[thm]{Corollary}
\newtheorem{prop}[thm]{Proposition}
\newtheorem{prop-ex}[thm]{Example}
\theoremstyle{definition}
\newtheorem{defn}[thm]{Definition}

\numberwithin{thm}{section}
\numberwithin{equation}{section}

\newcommand{\ml}[2]{\begin{multline}\label{#1}#2 \end{multline}}
\newcommand{\ga}[2]{\begin{gather}\label{#1}#2 \end{gather}}
\newcommand{\Pic}{{\rm Pic}}

\newcommand{\Spec}{{\rm Spec \,}}

\newcommand{\sA}{{\mathcal A}}
\newcommand{\sB}{{\mathcal B}}
\newcommand{\sC}{{\mathcal C}}

\newcommand{\sH}{{\mathcal H}}

\newcommand{\sO}{{\mathcal O}}
\newcommand{\sP}{{\mathcal P}}

\newcommand{\sT}{{\mathcal T}}

\newcommand{\A}{{\mathbb A}}

\newcommand{\G}{{\mathbb G}}

\newcommand{\N}{{\mathbb N}}

\newcommand{\Z}{{\mathbb Z}}

\begin{document}
\title[Surface singularities]{Surface singularities dominated by smooth varieties}
\author{H\'el\`ene Esnault}
\address{
Universit\"at Duisburg-Essen, Mathematik, 45117 Essen, Germany}
\email{esnault@uni-due.de}
\author{Eckart Viehweg}
\address{Universit\"at Duisburg-Essen, Mathematik, 45117 Essen, Germany}
\email{viehweg@uni-due.de}
\date{January 29, 2010}
\thanks{Partially supported by  the DFG Leibniz Preis, the SFB/TR45, the ERC Advanced Grant 226257}
\begin{abstract} We give a version in characteristic $p>0$ of Mumford's theorem  characterizing a smooth complex germ of surface $(X,x)$ by the triviality of the topological fundamental group of $U=X\setminus \{x\}$. 
\end{abstract}
\maketitle
\section{Introduction} Let $(X,x)$ be a 2-dimensional normal complex analytic germ. Let $U=X\setminus \{x\}$. Mumford  (\cite{Mum}) showed the celebrated theorem 
\begin{thm}[Mumford] $(X,x)$ if smooth if and only if the topological fundamental group of $U$ is trivial. 
 \end{thm}
This is a remarkable theorem which connects a topological notion to a 
scheme-theoritic one. His theorem has been a bit refined by Flenner \cite{Fle} who showed that in fact, the conclusion remains true if one replaces the topological by the \'etale fundamental group of $U$, that is by its profinite completion. Then one can replace the analytic germ by a complete or henselian germ over an algebraically closed  field $k$ of characteristic $0$. 

 If $k$ is an algebraically closed  field $k$ of characteristic $p>0$, Mumford himself observed that the theorem is no longer true. As an example, while in characteristic $0$,  the singularity $z^2+xy$ is the quotient of $\widehat{\A}^2$, the completion of $\A^2$ at the origin,  by the group $\Z/2$ acting  via ${\rm diag}(-1,-1)$, in characteristic $2$, it is the quotient of $\widehat{\A}^2$ by $\mu_2=\Spec k[t]/(t^2-1)$ acting via ${\rm diag}(t,t)$. Thus $\pi^{\rm et}(U)=\pi^{\rm et}(\widehat{\A}^2\setminus \{0\})=0$, yet $z^2+xy$ is not smooth.  

Artin asked in \cite{Ar3} whether, if $\pi^{{\rm et}}(U)$ is finite, there is always a finite morphism $ \widehat{\A}^2\to X$. He shows this if $(X,x)$ is a rational double point {\it loc.cit.}.

The purpose of this note is to give an answer to a similar question where one replaces the \'etale fundamental group by the Nori one. Strictly speaking, Nori
in \cite[Chapter~II]{N} defined his fundamental group-scheme for irreducible reduced schemes endowed with a rational point. But as $U$ has no rational point,  one has to modify a tiny bit Nori's construction to make it work. This is done in subsection \ref{ss_loc}. While the \'etale fundamental group of $X$ is trivial, Nori's one isn't. So the right notion of Nori's fundamental group is a relative one denoted by $\pi_{{\rm loc}}(U,X,x)$ (see Lemma \ref{lem2.5}). Roughly speaking, it measures the torsors on $U$ under a finite flat $k$-group-scheme $G$ which do not come by restriction from a torsor on $X$. We show (Theorem \ref{thm4.2}) that if $\pi_{{\rm loc}}^N(U,X,x)$  is finite, then $(X,x)$ is a rational singularity, and if $\pi_{{\rm loc}}^N(U,X,x)=0$, then there is a finite morphism $f: \widehat{\A}^2\to X$.

This note relies on discussions the authors had during the Christmas break 2009/10 in Ivry. They have been written down by H\'el\`ene in the night  when Eckart died, as a despaired sign of love.  

\section{Local Nori Fundamental Groupscheme}
\subsection{Nori's construction}
Let $U$ be a scheme defined over a  field $k$, endowed with a rational point $u\in U(k)$. 
In \cite[Chapter~II]{N} Nori constructed the fundamental group-scheme $\pi^N(U,u)$. Let $\sC(U,u)$ be the following category. The objects   are triples $(h:V\to U, G, v)$ where $G$ is a finite $k$-group-scheme, $h$ is a $G$-principal bundle and $v\in V(k)$ with  $h(v)=u$. Recall \cite[Chapter~I,2.2]{N} that a $G$-principal bundle
$h: V\to U$ is a flat morphism, together with a group action $G\times_k V\xrightarrow{\bullet} V$ such that $V\times_k G\xrightarrow{(1,\bullet)} V\times_U V$ is an isomorphism.  
  Then ${\rm Hom}\big((h_1:V_1\to U, G_1, v_1),(h_2:V_2\to U, G_2, v_2)\big)$ consists of the $U$-morphisms $f: V_1\to V_2$ which are compatible with the principal bundle structure. 

The objects of  the ind-category $\sC^{{\rm ind}}(U,u)$ associated to $\sC(U,u)$ are triples $(h:V\to U, G, v)$ where $G=\varprojlim_\alpha G_\alpha$ is a prosystem of finite $k$-group-schemes $G_\alpha$, $h=\varprojlim_\alpha h_\alpha, h_\alpha: V_\alpha\to U$, is a pro-$G$-principal bundle and $v=\varprojlim_\alpha v_\alpha \in Y(k)$ is a pro-point with  $h(v)=u$. The morphisms are the ind-morphisms $V_1\to V_2$ over $U$ which are compatible with the principal bundle structure and such that $f(v_1)=v_2$. 

Then $(U,u)$ has a fundamental group-scheme $\pi^N(U,u)$, which is then a $k$-profinite group-scheme,  if by definition \cite[Chapter~II, Definition~1]{N} there is a $(\frak{h}:W\to U, \pi^N(U,u),w)\in 
\sC^{\rm ind}(U,u)$ with the property that for any $(h: V\to U, G, v)\in \sC^{\rm ind}(U,u)$, there is a unique map 
$(\frak{h}:W\to U, \pi^N(U,u),w)\to (h: V\to U, G, v)$
 in $\sC^{{\rm ind}}(U,u)$.   

Nori shows \cite[Chapter~II, Lemma~1]{N} that if $G_1, G_2, G_0$ are three finite $k$-group-schemes, $h_i: V_i\to U$ are $G_i$-principal bundles, and $f_i: V_i\to V_0, i=1,2$ are principal bundle $U$-morphisms, then $V_1\times_{V_0} V_2\to Z$ is a principal bundle under $G_1\times_{G_0} G_2$, where $Z\subset U$ is a closed subscheme (no reference to the base point here). Then he shows that $(U,u)$ has a fundamental group-scheme if and only if $Z=U$ for all $(h_i: V_i\to U, G_i, y_i), f_i\in \sC(U,u)$
and he concludes \cite[Chapter~II, Proposition~2]{N} that if $U$ is reduced and irreducible, then $(U,u)$ has a fundamental group-scheme. 
\subsection{Local Nori fundamental group-scheme} \label{ss_loc}
Let $k$ be a field, let $A$ be a complete normal local $k$-algebra with maximal ideal $\frak{m}$ and residue field $k$. We define $X=\Spec A$ and $U=X\setminus \{x\}$, where $x\in X(k)$ is the rational point associated to $\frak{m}$. So in particular, $U(k)=\emptyset$, and we have to slightly modify Nori's construction to define the group-scheme of $U$.

Let $G$ be a finite $k$-group-scheme, and let $h: V\to U$ be a $G$-principal bundle. Recall from \cite[Corollaire~6.3.2, Proposition~6.3.4]{EGA2} that the {\it integral closure} $\tilde{h}: Y\to X$ of $h$ is the {\it unique}  extension $\tilde{h}: Y\to X$ of $h$ such that $Y=\Spec B$, $B$ is the integral closure of  $A$ in $j_*h_*\sO_V$, where $j: U\to X$ is the open embedding. Then $\tilde{h}$ is finite. In particular, if $h_i: V_i\to U$ are principal bundles under the finite $k$-group-schemes $G_i$, and $f: V_1\to V_2$ is a $U$-morphism which respects the principal bundle structures, then it extends uniquely to a $X$-morphism $\tilde{f}: Y_1\to Y_2$, which is then finite.
We can now mimic Nori's construction. 
\begin{defn} \label{defn2.1}
The objects of the category $\sC_{{\rm loc}}(U,x)$ are triples $(h: V\to U, G, y)$ where $G$ is a finite $k$-group-scheme, $y\in Y(k)$ with $\tilde{h}(y)=x$, where $\tilde{h}: Y\to X$ is the integral closure of $h$ . The morphisms ${\rm Hom} \big((h_1: V_1\to U, G_1, y_1) \to (h_2: V_2\to U, G_2, y_2)\big)$ consist of $U$-morphisms $f: V_1\to V_2$ which respect the principal bundle structure and such that $\tilde{f}(y_1)=y_2$.  

The objects of the ind-category $\sC^{{\rm ind}}_{{\rm loc}}(U,x)$ associated to $\sC_{{\rm loc}}(U,x)$ are triples $(h: V\to U, G, y)$ where $G=\varprojlim_\alpha G_\alpha$ is a pro-system of finite $k$-group-schemes,  
 $h=\varprojlim_\alpha h_\alpha, h_\alpha: V_\alpha\to U$, is a pro-$G$-principal bundle, and $y=\varprojlim_\alpha y_\alpha \in \varprojlim_\alpha Y_\alpha(k)$ is a pro-point in the integral closure of $V_\alpha$ mapping to $x$. 

One says that $(U,x)$ has a {\it local fundamental group-scheme} $\pi^N_{{\rm loc}}(U,x)$, which is then a $k$-profinite group-scheme,  if  there is a $(\frak{h}:W\to U, \pi^N_{{\rm loc}}(U,x),z)\in 
\sC^{\rm ind}_{{\rm loc}}(U,x)$ with the property that for any $(h: V\to U, G, v)\in \sC^{\rm ind}_{{\rm loc}}(U,x)$, there is a unique map 
$(\frak{h}:W\to U, \pi^N_{{\rm loc}}(U,x),z)\to (h: V\to U, G, y)$
 in $\sC^{{\rm ind}}_{{\rm loc}}(U,x)$.

\end{defn}
\begin{prop} \label{prop2.3}
 If $X$ is reduced and irreducible, then $(U,x)$ has a {\it local fundamental group-scheme} $\pi^N_{{\rm loc}}(U,x)$.
\end{prop}
\begin{proof}
 As explained above, the condition on $X$ implies that if
$f_i: (h_i: V_i\to U, G_i, y_i)\to (h_0: V_0\to U, G_0, y_0)$ is a morphism in $\sC_{{\rm loc}}(U,x)$, then $(V_1\times_{V_0} V_2\to U, G_1\times_{G_0} G_2, y_1\times_{y_0} y_2)\in \sC_{{\rm loc}}(U,x)$, so 
as in \cite[Chapter~II,p.87]{N}, the prosystem $\varprojlim_\alpha  (h_\alpha: V_\alpha\to U, G_\alpha, y_\alpha)$ over all objects $ (h_\alpha: V_\alpha\to U, G_\alpha, y_\alpha) $ of $\sC_{{\rm loc}}(U,x)$ is well defined. So $\pi^N_{{\rm loc}}(U,x)=\varprojlim_\alpha G_\alpha$.
\end{proof}
There is a restriction functor $\rho: \sC(X,x)\to \sC_{{\rm loc}}(U,x)$ which sends $(h: Y\to X, G, y)$ to its restriction $(h_U: Y\times_X U\to U, G, y)$, as the integral closure of $X$ in $Y\times_X U$ is $Y$. This defines the $k$-group-scheme homomorphism $$\rho_*: \pi^N_{{\rm loc}}(U,x)\to \pi^N(X,x).$$ 
\begin{prop} \label{prop2.3}
 The homomorphism $\rho$ is faithfully flat. 
\end{prop}
\begin{proof}
 Faithful flatness of $\rho$ means that if 
$(h: Y\to X, G, y)\in \sC(X,x)$ is such that $(Y_U\to, G, y) \to (U, \{1\},x)$ factors through $(\ell: V\to U, H, y) \in \sC_{{\rm loc}}(U,x)$, where $Y_U=Y\times_X U$, then
necessarily $(\ell: V\to U, H, y) =\rho(\ell_X: Z\to X, H, y)$ for some $(\ell_X: Z\to X, H, y) \in \sC(X,x)$.
  Let $K={\rm Ker}(G\to H)$. Since $K$ is a $k$-subgroup-scheme of $G$, it acts on $Y$. We define $Z$ to be $Y/K$.  By definition, $Z_U=V$. The compositum $h: Y\to Z\to X$ is a $G$-torsor. 
The  embedding $Y\times_Z Y\subset Y\times_X Y$ is closed, and while restricted to $U$, it is described as $Y_U \times_k K\subset Y_U\times_k G$. Thus $Y\times_ZY$ contains the closure  of $Y_U\times_k K$ in $Y\times_k G$, that is $Y\times_k K$.  Thus $Y\times_k K$ consists of connected components of $Y\times_Z Y$ and moreover, if there is another connected component, it lies in $\{y\}\times_Z Y=\Spec k$. Thus $Y\times_Z Y\cong_k Y\times_k K$ and $Y\to Z$ is a $K$-torsor. This finishes the proof.

\end{proof}
We denote by $\pi^{{\rm et}}(U,x)$ the \'etale proquotient of $\pi^N_{{\rm loc}}(U,x)$. From now on, we assume 
$k=\bar k$. Then $\pi^{{\rm et}}(U,x)$ is identified with $\pi^{\rm et}(U,\eta)$ where $\eta\to U$ is a geometric generic point and $\pi^{\rm et}(U,\eta)$ is Grothendieck's \'etale fundamental group. The \'etale proquotient of $\pi^N(X,x)$ is identified with Grothendieck's fundamental group based at $x$, and is trivial by Hensel's lemma, as $A$ is complete. If $\ell$ is a prime number (including $p$), we denote by $\pi^{{\rm et, ab,}\ell}(U,x)$ the maximal pro-$\ell$-abelian quotient of $\pi^{{\rm et}}(U,x)$. 
\begin{defn}
 One defines $\pi^N_{\rm loc}(U,X,x)={\rm Ker}\big( \pi^N_{{\rm loc}}(U,x)\xrightarrow{\rho} \pi^N(X,x)\big)$.
\end{defn}
From the discussion, we see
\begin{lem} \label{lem2.5}
  The compositum  $\pi^N_{{\rm loc}}(U,X,x)\to \pi^{\rm et}(U,x)$ is surjective. In particular, if $\pi^N_{{\rm loc}}(U,X,x)$ is a finite $k$-group-scheme, $\pi^{\rm et}(U,x)$ is a finite group.
\end{lem}

\section{Construction and elementary properties of the Picard scheme for surface singularities} 
 Let $k$ be a field, perfect if of characteristic $p>0$, let $A$ be a complete normal local $k$-algebra with maximal ideal $\frak{m}$,  $X=\Spec A$ and $U=X\setminus \{x\}$, where $x\in X(k)$ is the rational point associated to $\frak{m}$.
In \cite[Expos\'e~XIII,Section~5]{SGA2} Grothendieck initiated the construction of a pro-system of  locally algebraic $k$-group-schemes $G_n$ and a canonical isomomorphism $G(k)=\Pic(U)$ with $G(k)=\varprojlim_n G_n(k)$. This construction is performed in \cite{Lip} (see overview in \cite[p.~273]{Kl}) and relies on Mumford's basic idea \cite[Section~2]{Mum} to use a desingularization of $X$, if it exists, so in characteristic $0$ or if ${\rm dim}_k X\le 2$ if $k$ has characterisistic $p>0$. We now summarize  the construction and the elementary properties under the assumptions
\begin{itemize}
 \item[1)] $X$ is normal
\item[2)] ${\rm dim}_k X=2.$
\end{itemize}
Let $\sigma: \tilde{X} \to X$ be a  desingularization such that $\sigma^{-1}(x)_{{\rm red}}=\cup_i D_i$ is a strict normal crossings divisor and all components $D_i$ are $k$-rational. There is linear combination $D=\sum_i m_iD_i$ with all $m_i\ge 1$ such that $\sO_{\tilde{X}}(-D)$ is relatively ample. 
We define $\tilde{X}_n$ to be scheme $\cup_i D_i$ with structure sheaf $\sO_{\tilde{X}}/\sO_{\tilde{X}}(-(n+1)D)$, so $\tilde{X}_0=D$, and we also define $D_{{\rm red}}$ with structure sheaf $\sO_{\tilde{X}}/\sO_{\tilde{X}}(-\sum_i D_i)$.  
 Then the functors $\sP ic(\tilde{X}_n/k)$ and $\sP ic(D_{ {\rm red}}/k)$, taken as a Zariski, an \'etale or a fppf functor, are representable by locally algebraic $k$-group-schemes $\Pic(\tilde{X}_n/k)$ and $\Pic(D_{{\rm red}}/k)$, so $\Pic(\tilde{X}_n)=\Pic(\tilde{X}_n/k)(k), \ \Pic(D_{ {\rm red}})=\Pic(D_{ {\rm red}}/k)(k)$ (see \cite[p.~273]{Kl}, \cite[Theorem~1.2]{Lip}). On the other hand, for all $n\ge 0$, and all $k$-algebras $R$, one has $\Pic(\tilde{X}_n\otimes_k R)=H^1(\tilde{X}_n\otimes_k R, \sO^\times)$.  As the relative dimension of $\sigma$ is $1$, this implies that the transition homomorphisms $\Pic(\tilde{X}_{n+1})\to \Pic(\tilde{X}_n)\to \Pic(\tilde{X}_0)\to \Pic(D_{ {\rm red}})$ are all surjective, and that ${\rm Ker}\big( \Pic(\tilde{X}_{n+1})\to \Pic(\tilde{X}_n)\big)=H^1\big(\tilde{X}_0, \sO_{\tilde{X}_0}(-(n+1)D)\big)$. Since $-D$ is a relatively ample divisor on $\tilde{X}$, there is a $n_0\ge 0$ such that the transition homomorphisms $\Pic(\tilde{X}_n) \to \Pic( \tilde{X}_{n_0})$ are all constant for $n\ge n_0$. Since the $1$-component $\Pic^0(D_{{\rm red}})$
of  $\Pic(D_{{\rm red}})$ is a semi-abelian variety, so in particular smooth, and the fibers $\Pic(\tilde{X}_n) \to \Pic(D_{{\rm red}})$ are affine \cite[p.~9,Corollaire]{Oo},  
 $\Pic( \tilde{X}_{n_0})$ is smooth. One defines 
\ga{3.1}{\Pic(\tilde{X})= \Pic( \tilde{X}_{n_0}).}
 It is thus a locally algebraic smooth $k$-group-scheme. 
It is an extension of $\oplus_i \Z[D_i]$ by its $1$-component.
Its $1$-component $\Pic^0(\tilde{X})\subset \Pic(\tilde{X}) $ is an extension of a semiabelian variety by smooth, connected commutative unipotent algebraic group over $k$.

Let $\langle D\rangle\subset \Pic(\tilde{X})$ be the subgroup-scheme  spanned by those divisors with support in $D$. (In fact, $\langle D\rangle$ injects into $\Pic(D_{ {\rm red}})$ via the surjection $\Pic(\tilde{X})\to \Pic(D_{ {\rm red}})$).
 It is a discrete subgroup-scheme. One sets 
\ga{3.2}{\Pic(U)=\Pic(\tilde{X})/\langle D\rangle.} 
The Zariski tangent space at $1$  is 
\ga{3.3}{H^1(\tilde{X}, \sO_{\tilde{X}})=H^1(\tilde{X}_n, \sO_{\tilde{X}_n})= {\rm Ker}\big(\Pic(\tilde{X}_n[\epsilon])\to \Pic(\tilde{X}_n)    \big)} for $n\ge n_0$, where $\tilde{X}_n[\epsilon]:=\tilde{X}_n\times_k k[\epsilon]/(\epsilon^2)$.  Since $\Pic(\tilde{X})$ is smooth,  
\ga{3.4}{{\rm dim}_kH^1(\tilde{X}, \sO_{\tilde{X}})={\rm dim} \ \Pic^0(\tilde{X})=\Pic^0(U).} 
The last equality comes from the fact that $\langle D\rangle\subset \Pic(\tilde{X})$ is a discrete \'etale subgroup.

Recall that the surface singularity $(X,x)$ is said to be {\it rational} if $H^1(\tilde{X}, \sO_{\tilde{X}})=0$. The definition does not depend on the choice of the resolution $\sigma: \tilde{X}\to X$ of singularities of $(X,x)$. 

One has
\begin{lem} \label{lem3.1}
The following conditions are equivalent.
\begin{itemize}
\item[1)] 
The surface singularity $(X,x)$ is rational. 
\item[2)] $\Pic^0(\tilde{X})=0$.
\item[3)] $\Pic(U)$ is finite. 
\end{itemize}
\end{lem}
\begin{proof}
 The equivalence of 1) and 2) is given by \eqref{3.4}. As $\langle D\rangle \subset \Pic(\tilde{X})$ is discrete, the definition \eqref{3.2} shows that 3) implies 2). Vice-versa, assume 2) holds. Then $\Pic(\tilde{X})$ is a discrete group of finite type. Let $L\in \Pic(\tilde{X})$. Since the intersection matrix $(D_i\cdot D_j)$ is negative definite (but not necessarily unimodular), there is a $m\in \N\setminus \{0\}$ such that $L^{\otimes m}\in \langle D\rangle \subset \Pic(\tilde{X})$. Thus any $L\in \Pic(\tilde{X})$ has finite order in $\Pic(U)$. Since $\Pic(\tilde{X})$ is of finite type, this shows 3). 
 
\end{proof}
\section{The Theorems}
Throughout this section, we assume  $k$ to be a field, perfect if of characteristic $p>0$,  $A$ to be a complete normal local $k$-algebra with maximal ideal $\frak{m}$, of Krull dimension $2$ over $k$. We set   $X=\Spec A$,  $U=X\setminus \{x\}$, where $x\in X(k)$ is the rational point associated to $\frak{m}$. We say $(X,x)$ is a {\it surface singularity} over $k$. We denote by $\sigma: \tilde{X}\to X$ a desingularization such that $\sigma^{-1}(x)_{{\rm red}}=\cup_i D_i$ is a strict normal crossings divisor. 
We define $H^i(Z,\Z_\ell(1)):=\varprojlim_n H^i(Z, \mu_{\ell^n})$ for a $k$-scheme $Z$. 
\begin{thm} \label{thm4.1}
 Let $(X,x)$ be a surface singularity over an algebraically closed field $k$.  The following conditions are equivalent
\begin{itemize}
\item[1)] $H^1(\tilde{X},\Z_\ell(1))=0$.
\item[2)] $H^1(\tilde{U}, \Z_\ell(1))=0$. 
 \item[3)] There is a prime number $\ell$,   different from $p$ if ${\rm char}(k) =p>0$, such that $\pi^{\rm et,ab,\ell}(U,x)$ is finite.
\item[4)] For all prime numbers $\ell$,  $\pi^{\rm et,ab,\ell}(U,x)$ is finite and if ${\rm char}(k)=p>0$, then $\pi^{\rm et,ab,\ell}(U,x)=0$.
\item[5)] $\Pic^0(\tilde{X})=\Pic^0(U)$ is a smooth, connected commutative unipotent algebraic group-scheme over $k$.  
\item[6)] $D$ is a tree of $\mathbb{P}^1$s. 
\item[7)] $\Pic^0(D_{{\rm red}})=0$.
\end{itemize}
\end{thm}
\begin{proof}
We firt make general remarks. 
For any surface singularity, one has   the localization sequence 
\ml{4.1}{  H^1(\tilde{X}, \Z_\ell(1))\to H^1(U, \Z_\ell(1))\to\\ H^2_{D_{{\rm red}}}(\tilde{X}, \Z_\ell(1))\to H^2(\tilde{X}, \Z_\ell(1)) \to H^2(U, \Z_\ell(1))\to H^3_{D_{{\rm red}}}(\tilde{X}, \Z_\ell(1)) \to H^3(\tilde{X}, \Z_\ell(1)).}
By purity \cite[Theorem~2.1.1]{Fu}, the restriction map 
$H^1(\tilde{X}, \Z_\ell(1))\to H^1(U, \Z_\ell(1))$ is injective,  and
$H^2_{D_{{\rm red}}}(\tilde{X}, \Z_\ell(1))=\oplus_i \Z_\ell\cdot [ D_i]$. By base change, $H^i(\tilde{X}, \Z_\ell(1))= H^i(D_{{\rm red}}, \Z_\ell(1))$. Thus this group is $0$ for $i\ge 3$, equal to  $\oplus_i \Z_\ell\cdot [ D_i]$ for $i=2$, and equal to $\Pic(D_{{\rm red}})[\ell]$ for $i=1$. In fact, since $H^2(D_{{\rm red}}, \Z_\ell(1))$ is torsion free, one has $\Pic(D_{{\rm red}})[\ell]=\Pic^0(D_{{\rm red}})[\ell]$, where $^0$ means of degree $0$ on each component $D_i$. 
Furthermore, by definition,   the map 
$\oplus_i \Z_\ell\cdot [ D_i] \to \oplus_i \Z_\ell\cdot [ D_i]$ is defined by $[D_i]\mapsto \oplus_j {\rm deg}\sO_{D_j}(D_i)$. Since the intersection matrix is definite, the map is injective, with finite torsion cokernel $\sT$. (This cokernel is 0 if and only if the intersection matrix is unimodular). Again by purity, 
$H^3_{D_{{\rm red}}}(\tilde{X}, \Z_\ell(1)) \subset \oplus_i H^1(D_i^0, \Z_\ell)$ where $D_i^0=D_i\setminus \cup_{j\neq i} D_i\cap D_j$. In particular, 
$H^3_{D_{{\rm red}}}(\tilde{X}, \Z_\ell(1))$ is torsion free. 
So we extract from \eqref{4.1} for any surface singularity the relations
\ga{4.2}{H^1(\tilde{X}, \Z_\ell(1))\to H^1(U, \Z_\ell(1))=\Pic(D_{{\rm red}})[\ell]=\Pic^0(D_{{\rm red}})[\ell]
}
and an exact sequence
\ga{4.3}{0\to \sT\to H^2(U, \Z_\ell(1))\to 
H^3_{D_{{\rm red}}}(\tilde{X}, \Z_\ell(1))\to 0}
with finite $\sT$ and torsion free $H^3_{D_{{\rm red}}}(\tilde{X}, \Z_\ell(1))$.
As $\Pic^0(D_{{\rm red}})$ is a semiabelian variety, we see that \eqref{4.2} implies that 1), 2) and 7) are equivalent conditions. 

From the exact sequence
\ga{4.4}{1\to \sO^\times_{D_{{\rm red}}}\to \oplus_i  \sO^\times_{D_i} \to \oplus_{i<j}k^\times_{D_i\cap D_j}\to 1}
one has that 6) and 7) are equivalent. Furthermore, from the structure of $\Pic(\tilde{X})$ explained in section 3, one has that  5) is equivalent to 7). 

We show that 2) is equivalent to 3). The condition 2) implies that $H^1(U, \mu_{\ell^n})\subset \sT$ for all $n\ge 0$, thus there are finitely many $\mu_{\ell^n}$ torsors on $U$. This shows 2) implies  3). On the other hand, if $\Pic^0(D_{{\rm red}})$ is not trivial, then $\Pic(D_{{\rm red}})[\ell]$ contains $\Z_\ell$. Thus $H^1(U, \Z_\ell(1))$ contains $\Z_\ell$ as well by  
\eqref{4.2}. Thus 3) implies 2). 

Since obviously 4)  implies 3), it remains to see that 3) implies 4). We assume 3). 
For any commutative finite $k$-group-scheme $G$, with Cartier dual $G'={\rm Hom}(G, \G_m)$, 
one has the exact sequence 
\ga{4.5}{0\to H^1(X, G')\to H^1(U, G')\to {\rm Hom}(G, \Pic(U))\to 0.} 
(See \cite[III,Th\'eor\`eme~4.1]{Boutot} and \cite[III,Corollaire~4.9]{Boutot} for the $0$ on the right, which we will use only on the proof of Theorem \ref{thm4.2}, as $k=\bar k$).
We apply it for $G=\Z/p^n$ for some $n\in \N\setminus \{0,1\}$.
Since $\Pic(U)$ is an extension of a discrete (\'etale) group by $\Pic^0(U)$ which is a product of $\G_a$s by 5),
one has ${\rm Hom}(\mu_{p^n}, \Pic(U))=0$. On the other hand, $A\xrightarrow{x\mapsto (x^{p^n}-x)}A$ is surjective, as $A$ is complete. Thus $H^1(U, \Z/p^n)=H^1(X,\Z/p^n)= 0$. This shows that 3) implies 4) and finishes the proof of the theorem.

\end{proof}
\begin{thm} \label{thm4.2}
 Let $(X,x)$ be a surface singularity over an algebraically closed field $k$. 
\begin{itemize}
\item[1)] If $\pi^N_{{\rm loc}}(U,X, x)$  is a finite group-scheme, $(X,x)$ is a rational singularity, in particular the dualizing sheaf $\omega_U$ has finite order.
\item[2)] If in addition, the order of $\omega_U$ is prime to $p$, then 
 there is $\big(h: V\to U, \pi^N(U,x),y\big) \in \sC_{{\rm loc}}(U,x)$  such that  the surface singularity $(Y,y)$ of the integral closure $\tilde{h}: Y\to X$ is a rational double point. 
\item[3)] If $\pi^N_{{\rm loc}}(U,X,x)=0$, then $(X,x)$ is a rational double point. 
\end{itemize}
\end{thm}
\begin{proof} We show 1). 
 If $\pi^N_{{\rm loc}}(U,X, x)$ is a finite group-scheme, then, by Lemma \ref{lem2.5},   the condition 3) of Theorem \ref{thm4.1} is fulfilled, thus $\Pic^0(\tilde{X})=\Pic^0(U)$ is a product of $\G_a$s. We apply \eqref{4.5} to $G=\Z/p^n$. If $\Pic^0(U)$ is not trivial, then ${\rm Hom}(\Z/p^n, \Pic(U))\neq 0$ for all $n\ge 0$. Thus $U$ admits nontrivial $\mu_{p^n}$-torsors for all $n\ge 1$, which do not come from $X$. This contradicts the finiteness of $\pi^N_{{\rm loc}}(U,X,x)$. Thus $\Pic^0(U)=\Pic^0(\tilde{X})=0$.  We apply  Lemma \ref{lem3.1} to finish conclude that $(X,x)$ is a rational singularity. Again by Lemma \ref{lem3.1}, all line bundles on $U$, in particular the dualizing sheaf $\omega_U$ of $U$, is torsion. This proves 1).
 
We show 2). 
 So there is a $M\in \N\setminus \{0\}$ such that $\omega^M_U\cong \sO_U$.  Choosing such a trivialization yields an $\sO_U$-algebra structure on $\sA=\oplus_0^{M-1} \omega^i_U$ and thus  a flat nontrivial $\mu_M$-torsor $h:V=\Spec_{\sO_U} \sA \to U$.  Since $(M,p)=1$, $h$ is \'etale, thus $(Y,y)$ is normal. In fact one has $Y=\Spec_{\sO_X} \sB$ where $\sB$ is the $\sO_X$-algebra $j_*\sA, j: U\subset X$. By duality theory, $h_*\omega_Y=\sH om_{\sO_X}(h_*\sO_Y, \omega_X)\cong_{\sO_X} h_*\sO_Y$. Let $y\in Y$ be the closed point of $Y$. Thus $(Y,y)$ is a Gorenstein normal surface singularity. On the other hand, since $h$ is a $\mu_M$-torsor, one has $\pi^N(V,y)\subset \pi^N(U,x)$, thus $\pi^N_{{\rm loc}}(V,Y,y)\subset \pi^N_{{\rm loc}}(U,X,x)$, and therefore is a finite $k$-group-scheme. Thus by 1) it is a rational singularity. Thus $(Y,y)$ is a Gorenstein rational singularity, thus is a rational double point (\cite{Durfee}). 

Now 3) follows directly from 2) as $\omega_U$ has then order $1$. 

\end{proof}
We now refer to \cite[Section~3]{Ar3} for the notation, and we go to Artin's list \cite[Section~4/5]{Ar3} to conclude using Theorem \ref{thm4.2} 3):
\begin{cor} \label{cor4.3}
 If $\pi^N_{\rm loc}(U,X,x)=0$, then $X$ admits a finite  morphism $f:\widehat{\A}^2\to X$. The morphism $f$ is the identity (i.e. $(X,x)$ is smooth) except possibly in the cases:
\begin{itemize}
 \item[1)] ${\rm char}(k)=2$, $E^1_8, E_8^3$
\item[2)] ${\rm char}(k)=3$, $ E_8^1$
\end{itemize}

\end{cor}

\bibliographystyle{plain}
\renewcommand\refname{References}

\end{document}